\documentclass{amsart}
\usepackage{graphicx} 
\usepackage{graphicx} 
\usepackage[utf8]{inputenc}
\usepackage[english]{babel}
\usepackage{amsmath} 
\usepackage{amssymb}
\usepackage{amsthm}
\usepackage{mathtools}
\usepackage{graphicx}
\graphicspath{ {c:/This PC/Documents/images/} }
\usepackage{tikz-cd}
\usepackage{MnSymbol}
\usepackage{amsthm}
\usepackage{comment}

\usepackage[backend=biber, style=numeric, sorting=none]{biblatex}
\AtBeginBibliography{\footnotesize}

\theoremstyle{plain}
\newtheorem{theorem}{Theorem}[section]

\newtheorem{lemma}[theorem]{Lemma}

\theoremstyle{definition}

\newtheorem*{remark}{Remark}

\usepackage[toc,page]{appendix}

\usepackage[top=30mm, left=35mm, right=35mm]{geometry}

\title{Commutators of n-cycles in the symmetric group}
\author[Philipp Bader]{Philipp Bader}
\address{University of Glasgow, UK}
\email{p.bader.1@research.gla.ac.uk} 


\begin{document}

The result presented in this article can be derived from the work of Aleš Vavpetič in [1]. We would like to thank Goulnara Arzhantseva and Aleš Vavpetič for pointing out the reference as well as for the subsequent discussion and explanations.\\

In our main theorem, we want to show that every even permutation $\rho$ in $A_n$ can be written as a commutator of two $n$-cycles. As Aleš Vavpetič pointed out to us, this follows in most cases by putting together Corollary 2.5, Corollary 2.7 and Lemma 2.9 in [1]. The only cases not covered in this way are when $\rho$ contains $3$-cycles or pairs of $2$-cycles. However, these cases also follow by a short computation and an induction argument similar to the one used in above Corollaries.\\

Since the proof we present is not conceptually different from what is done in [1], this article is therefore not intended for publication.\\

\newpage

\begin{abstract}
We show that for $n \ge 6$ every even permutation on $n$ symbols is the commutator of two $n$-cycles. More precisely, let $S_n$ be the symmetric group and $A_n$ the alternating group. Let $C(n) \subset S_n$ denote the conjugacy class of $n$-cycles and $[\cdot, \cdot]$ be the commutator of two permutations. We prove: The map $C(n) \times C(n) \to A_n, \ (\tau, \pi) \mapsto [\tau, \pi]$ is surjective for all $n \ge 6$.
\end{abstract}

\maketitle
\section{Introduction}

The commutator subgroup of a group is defined as the subgroup generated by all commutators. A product of two commutators need not be a commutator itself. Hence, it makes sense to ask whether the commutator subgroup of a given group consists entirely of commutators.\\

In 1951, Ore proved that the above holds for the finite symmetric groups by showing that every element of the alternating group can be written as a commutator in the symmetric group (\cite{O}). Furthermore, Ore proceeded to prove that for $n \ge 5$ every element in the alternating group $A_n$ is already a commutator in the alternating group itself. This led to the so called Ore conjecture which states that every element in a finite, non-abelian, simple group is a commutator. Over the years the conjecture was established for more simple groups and finally completely proven in 2010 by Liebeck, O'Brien, Shalev and Tiep (\cite{LOST}).\\

In this note, we reinforce Ore's original result by showing that for $n \ge 6,$ every element in the alternating group $A_n$ can be written as a commutator of two $n$-cycles. In the case where $n$ is odd, $n$-cycles are elements of $A_n$ themselves.\\ 

We start by setting up some notation:\\

Let $S_n$ be the symmetric group on $n$ symbols. For $\pi \in S_n$ we say that $\pi$ has cycle type $(k_1, ... , k_m)$ if $\pi$ consists of $m$ cycles that have length $k_1, ... , k_m$ respectively. Note that $k_1 + ... + k_m = n.$ The conjugacy classes of $S_n$ are determined by cycle type, i.e. two permutations $\pi, \tau \in S_n$ are conjugate if and only if they have the same cycle type. For $k_1, ... , k_m$ with $k_1 + ... + k_m = n,$ we denote the corresponding conjugacy class by $C(k_1, ... , k_m) \subset S_n.$ We refer to elements in $C(n)$ as \textit{$n$-cycles}.\\

Denote by $[\tau, \pi] := \tau^{-1}\pi^{-1}\tau\pi$ the commutator of $\tau, \pi \in S_n.$ Every commutator is an even permutation and hence an element of $A_n.$ It is a classical result that the commutator subgroup of $S_n,$ i.e. the group generated by all commutators, is $A_n$ for all $n \ge 2.$\\

Here, for $n \ge 6$ we prove the stronger result that every element of $A_n$ is a commutator of two $n$-cycles.

\begin{theorem}\label{commutator theorem}
    The map $$C(n) \times C(n) \to A_n, \ (\tau, \pi) \mapsto [\tau, \pi]$$ is surjective for all $n \ge 6.$    
\end{theorem}

\begin{remark}
    For $n = 2$ Theorem \ref{commutator theorem} is also true since $A_2$ is the trivial group. For $n \in \{3, 4, 5\}$ it fails however. For $n = 3,$ we have that $C(3) = \{\sigma, \sigma^{-1}\}$ with $\sigma = (1 \ 2 \ 3)$ and so any commutator of $3$-cycles is the identity. On the other hand however, $\sigma, \sigma^{-1} \in A_3.$ For $n = 4$ and $n = 5$ one can check that $C(2,2) \subset A_4$ and $C(1, 2, 2) \subset A_5$ respectively are not in the image. 
\end{remark}

We point out that the theorem also follows from Jeffreys' work in \cite{LJ} although there it is not stated in this form. Jeffreys was studying so called square tiled surface which can be described by permutations. The author motivated by similar questions came across this purely group theoretical problem. Lemma \ref{stitching together} below can be found in a more geometrical language in (\cite{LJ}, Lemma 3.1) and was proved independently by the author. The result and proof presented here are purely group theoretical and don't use any geometry.

\section{Proof of the theorem}

Let $\sigma_n := (1 \ 2 \ ... \ n) \in C(n).$ We refer to the $\sigma_n$ as the \textit{shift permutation}. Let $k_1, ... , k_m \in \mathbb{N}$ with $k_1 + ... + k_m = n.$ The conjugacy class $C(k_1, ... , k_m)$ is a subset of $A_n$ if and only if the tuple $(k_1, ... , k_m)$ contains an even number of even numbers. We refer to such conjugacy classes as \textit{even conjugacy classes}.\\

We denote the commutator map by $f.$ Note that it suffices to show that one element of each even conjugacy class is in the image of $f.$ For if $\rho \in \text{image}(f)$ and $\rho'$ is any other element in the same conjugacy class as $\rho,$ then there exist $\tau, \pi \in C(n)$ and $\phi \in S_n$ with
$$f(\tau, \pi) = \rho \text{ and } \rho' = \phi \rho \phi^{-1},$$

and hence we have $\tau' := \phi \tau \phi^{-1}, \, \pi' := \phi \pi \phi^{-1} \in C(n)$ with
$$f(\tau', \pi') = [\phi \tau \phi^{-1}, \phi \pi \phi^{-1}] = \phi [\tau, \pi] \phi^{-1} = \phi \rho \phi^{-1} = \rho',$$

and so $\rho' \in \text{image}(f).$\\

The above observation shows that it suffices to find examples of permutations $\tau, \pi \in C(n)$ whose commutator is contained in $C(k_1, ... , k_m)$ for every even conjugacy class $C(k_1, ... , k_m).$ If $[\tau, \pi] \in C(k_1, ... , k_m),$ then by simultaneously conjugating $\tau$ and $\pi$ we can assume that $\tau = \sigma_n,$ where $\sigma_n$ is the shift permutation. Hence, in order to prove the main theorem we want to prove the following reformulation of it:\\

\textbf{Claim}:
For every $n \ge 6$ and every even conjugacy class $C(k_1, ... , k_m) \subset S_n,$ there exists a $\pi \in C(n)$ with $[\sigma_n, \pi] \in C(k_1, ... , k_m).$\\

To prove the claim, we construct explicit examples in certain conjugacy classes which we will call \textit{irreducible} and show that it suffices to construct these examples. The latter follows from the following lemmata.\\ 

In the proofs, for two permutations $\alpha$ and $\beta$ on different letters, we use the notation $(\alpha) \, (\beta)$ to indicate the permutation that does both at the same time. For example, if $\alpha = (1 \, 2 \, 3)$ and $\beta = (4 \, 5) \, (6),$ then $(\alpha) \, (\beta) = (1 \, 2 \, 3) \, (4 \, 5) \, (6).$ 

\begin{lemma}\label{adding 1}
    Let $n \ge 2$ and $\pi \in C(n)$ with $\pi(1) = 2$ such that $[\sigma_n, \pi] \in C(k_1, ... , k_m).$ Then there exists $\pi_+ \in C(n+1)$ with $\pi_+(1) = 2$ such that 
    $$[\sigma_{n+1}, \pi_+] \in C(k_1, ... , k_m, 1).$$
\end{lemma}
\begin{proof}
Since $\pi$ is an $n$-cycle with $\pi(1) = 2,$ we can write $\pi = (1 \ 2 \ i \ ... \ j)$ where $i, j \in \{3, ... , n\}$ and between $i$ and $j$ we could have any order of the remaining numbers.\\

Define $\pi_+ := (1 \ \ 2 \ \ i \ \ ... \ \ j \ \ n+1) \in C(n+1)$ where between $i$ and $j,$ $\pi_+$ is the same as $\pi.$\\

Let $\alpha := [\sigma_n, \pi] \in C(k_1, ... , k_m).$ We claim that $[\sigma_{n+1}, \pi_+] = (\alpha) \, (n+1) \in C(k_1, ... , k_m, 1).$ To see this, first note that 
$$\sigma_{n+1}^{-1}\pi_+^{-1}\sigma_{n+1}\pi_+(n+1) = \sigma_{n+1}^{-1}\pi_+^{-1}\sigma_{n+1}(1) = \sigma_{n+1}^{-1}\pi_+^{-1}(2) = \sigma_{n+1}^{-1}(1) = n+1.$$

For any $k \in \{1, ... , n\}$ we need to check that $[\sigma_n, \pi](k) = [\sigma_{n+1}, \pi_+](k).$ This is straightforward for all $k$ apart from two cases, namely $k = j$ and $k$ with $\pi(k) = n.$\\

For $k = j$ we obtain:
$$[\sigma_{n+1}, \pi_+](k) = \sigma_{n+1}^{-1}\pi_+^{-1}\sigma_{n+1}(n+1) = \sigma_{n+1}^{-1}\pi_+^{-1}(1) = \sigma_{n+1}^{-1}(n+1) = n,$$
$$[\sigma_n, \pi](k) = \sigma_n^{-1}\pi^{-1}\sigma_n(1) = \sigma_n^{-1}\pi^{-1}(2) = \sigma_n^{-1}(1) = n.$$\\

For $k$ with $\pi(k) = n$ we obtain:
$$[\sigma_{n+1}, \pi_+](k) = \sigma_{n+1}^{-1}\pi_+^{-1}\sigma_{n+1}(n) = \sigma_{n+1}^{-1}\pi_+^{-1}(n+1) = \sigma_{n+1}^{-1}(j) = j-1,$$
$$[\sigma_n, \pi](k) = \sigma_n^{-1}\pi^{-1}\sigma_n(n) = \sigma_n^{-1}\pi^{-1}(1) = \sigma_n^{-1}(j) = j-1.$$
\end{proof}

\begin{lemma}\label{stitching together}
    Let $n_1 , n_2 \ge 2$ and $n := n_1 + n_2.$ Let $\pi_1 \in C(n_1), \, \pi_2 \in C(n_2)$ with $\pi_1(1) = \pi_2(1) = 2$ such that $[\sigma_{n_1}, \pi_1] \in C(k_1, ... , k_m)$ and $[\sigma_{n_2}, \pi_2] \in C(l_1, ... , l_s).$ Then there exists $\pi \in C(n)$ with $\pi(1) = 2$ such that 
    $$[\sigma_n, \pi] \in C(k_1, ... , k_m, l_1, ... , l_s).$$
\end{lemma}
\begin{proof}
Analogously to the proof of the previous lemma, we write $\pi_1 = (1 \ 2 \ i_1 \ ... \ j_1),$ $\pi_2 = (1 \ 2 \ i_2 \ ... \ j_2),$ where $i_1, j_1 \in \{3, ... , n_1\}$ and $i_2, j_2 \in \{3, ... , n_2\}.$\\ 

Let $\alpha := [\sigma_{n_1}, \pi_1] \in C(k_1, ... , k_m), \ \beta := [\sigma_{n_2}, \pi_2] \in C(l_1, ... , l_s)$ and define 
$$\pi := (1 \ \ 2 \ \ i_1 \ \ ... \ \ j_1 \ \ n_1+1 \ \ n_1+2 \ \ n_1+i_2 \ \ ... \ \ n_1+j_2).$$

Here, between $i_1$ and $j_1,$ $\pi$ is the same as $\pi_1$ and between $n_1+i_2$ and $n_1+j_2,$ $\pi$ is the same as $\pi_2$ after adding $n_1$ to each element. By definition $\pi \in C(n)$ and we want to show that $[\sigma_n, \pi] = (\alpha)(\beta + n_1) \in C(k_1, ... , k_m, l_1, ... , l_s),$ where the notation $\beta+n_1$ means to add $n_1$ to each element of $\beta.$\\

In particular, we want to show that $[\sigma_n, \pi](k) = [\sigma_{n_1}, \pi_1](k)$ for every $k \in \{1, ... , n_1\}$ and $[\sigma_n, \pi](l+n_1) = [\sigma_{n_2}, \pi_2](l)+n_1$ for every $l \in \{1, ... , n_2\}.$ This is straighforward for all $k$ and $l$ apart from two cases respectively, namely $k = j_1$ and $k$ with $\pi_1(k) = n_1,$ as well as $l = j_2$ and $l$ with $\pi_2(l) = n_2.$ For these cases, we carry out the computations:\\

For $k = j_1$ we obtain:
$$[\sigma_n, \pi](k) = \sigma_n^{-1}\pi^{-1}\sigma_n(n_1+1) = \sigma_n^{-1}\pi^{-1}(n_1+2) = \sigma_n^{-1}(n_1+1) = n_1,$$
$$[\sigma_{n_1}, \pi_1](k) = \sigma_{n_1}^{-1}\pi_1^{-1}\sigma_{n_1}(1) = \sigma_{n_1}^{-1}\pi_1^{-1}(2) = \sigma_{n_1}^{-1}(1) = n_1.$$\\

For $k$ with $\pi_1(k) = n_1$ we obtain:
$$[\sigma_n, \pi](k) = \sigma_n^{-1}\pi^{-1}\sigma_n(n_1) = \sigma_n^{-1}\pi^{-1}(n_1+1) = \sigma_n^{-1}(j_1) = j_1-1,$$
$$[\sigma_{n_1}, \pi_1](k) = \sigma_{n_1}^{-1}\pi_1^{-1}\sigma_{n_1}(n_1) = \sigma_{n_1}^{-1}\pi_1^{-1}(1) = \sigma_{n_1}^{-1}(j_1) = j_1-1.$$\\

For $l = j_2$ we obtain:
$$[\sigma_n, \pi](l+n_1) = \sigma_n^{-1}\pi^{-1}\sigma_n(1) = \sigma_n^{-1}\pi^{-1}(2) = \sigma_n^{-1}(1) = n = n_2+n_1,$$
$$[\sigma_{n_2}, \pi_2](l) = \sigma_{n_2}^{-1}\pi_2^{-1}\sigma_{n_2}(1) = \sigma_{n_2}^{-1}\pi_2^{-1}(2) = \sigma_{n_2}^{-1}(1) = n_2.$$\\

For $l$ with $\pi_2(l) = n_2$ we obtain:
$$[\sigma_n, \pi](l+n_1) = \sigma_n^{-1}\pi^{-1}\sigma_n(n) = \sigma_n^{-1}\pi^{-1}(1) = \sigma_n^{-1}(j_2+n_1) = j_2-1 + n_1,$$
$$[\sigma_{n_2}, \pi_2](l) = \sigma_{n_2}^{-1}\pi_2^{-1}\sigma_{n_2}(n_2) = \sigma_{n_1}^{-1}\pi_1^{-1}(1) = \sigma_{n_1}^{-1}(j_2) = j_2-1.$$
\end{proof}

In regard of the above lemmata, given an even conjugacy class $C(k_1, ... , k_m),$ we want to split the tuple $(k_1, ... , k_m)$ into smaller pieces for which we can construct the desired permutations for the corresponding conjugacy classes. A straightforward idea would be to split it into odd singletons and pairs of even numbers, construct permutations for each corresponding conjugacy class and use Lemma \ref{stitching together} to obtain a permutation for the whole conjugacy class. More precisely, this means that we would only need to construct permutations $\pi$ whose commutator with $\sigma_n$ is contained in $C(n)$ for $n$ odd or $C(2k, 2l)$ for $k,l \ge 1.$ However, this does not work, since for $C(3)$ or $C(2,2)$ there does not exist such a desired permutation. Hence, we need to consider more conjugacy classes. The following list will be referred to as \textit{irreducible conjugacy classes}:

$$\begin{array}{cc}
    C(n) & n \ge 5 \text{ odd} \\[0.5ex]
    C(2k, 2l) & k, l \ge 1, \text{ except } k = l = 1 \\[0.5ex]
    C(n, 3) & n \ge 1 \text{ odd} \\[0.5ex]
    C(2k, 2l, 3) & k, l \ge 1 \\[0.5ex]
    C(n, 2, 2) & n \ge 3 \text{ odd} \\[0.5ex]
    C(2k, 2, 2, 2) & k \ge 1 \\[0.5ex]
\end{array}$$

We show that it suffices to consider the above list. The reader might want to compare this to the corresponding list of strata in (\cite{LJ}, Section 4).

\begin{lemma}\label{irreducible suffice}
    Let $n \ge 6$ and let $C(k_1, ... , k_m) \subset S_n$ be an even conjugacy class. Assume that for every irreducible conjugacy class $C$ there exists a $\pi_C$ with $\pi_C(1) = 2$ such that its commutator with the corresponding shift permutation lies in $C$. Then there exists $\pi \in C(n)$ with $[\sigma_n, \pi] \in C(k_1, ... , k_m).$
\end{lemma}
\begin{proof}
    In this proof whenever we say \textit{a permutation for the conjugacy class} $C(l_1, ... , l_s),$ we mean an $m$-cycle $\tau,$ where $m := l_1+...+l_s,$ with the properties $\tau(1) = 2$ and $[\sigma_m, \tau] \in C(l_1, ... , l_s).$ Furthermore, we say that $C(l_1, ... , l_s)$ \textit{contains a number} $l$ if $l_j = l$ for some $j = 1, ... , s.$\\

    Our goal is to show that there is a $\pi$ for $C(k_1, ... , k_m).$ If $C(k_1, ... , k_m)$ does not contain the numbers $1, 2$ and $3,$ then by the assumption we can find permutations for the irreducible conjugacy classes $C(k_i)$ with $k_i$ odd and $C(k_j, k_{j'})$ with $k_j, k_{j'}$ even. Then, by Lemma \ref{stitching together} we can also find a $\pi$ for $C(k_1, ... , k_m).$\\

    So, we need to consider the cases where $C(k_1, ... , k_m)$ contains $1, 2$ or $3.$ First, assume that $C(k_1, ... , k_m)$ contains $1,$ but doesn't contain $2$ or $3.$ Let $k_i = 1$ for some $i.$ Then, if we can find a permutation for $C(k_1, ... , k_{i-1}, k_{i+1}, ... , k_m),$ by Lemma $\ref{adding 1}$ we can also find one for $C(k_1, ... , k_m).$ Inductively by this process we can remove all numbers equal to $1$ and be left with the previous case, so again we can find a $\pi$ for $C(k_1, ..., k_m).$\\ 

    Now, assume that $C(k_1, ... , k_m)$ contains $3,$ but doesn't contain $1$ or $2.$ Consider the permutations $(1 \ 2 \ 4 \ 3 \ 6 \ 5)$ and $(1 \ 2 \ 4 \ 9 \ 5 \ 7 \ 8 \ 6 \ 3).$ Their commutator with the corresponding shift permutations is contained in $C(3,3)$ and $C(3,3,3)$ respectively. So, by Lemma \ref{stitching together}, we can find permutations for any conjugacy class of the form $C(3, ... , 3),$ that is the conjugacy class containing only and at least two $3$'s. Therefore, if $C(k_1, ... , k_m)$ contains more than one $3,$ we can find a permutation for the conjugacy class $C(3,...,3)$ with the number of $3$'s that $C(k_1, ... , k_m)$ contains and using Lemma \ref{stitching together} reduce to the case of having no $3$'s. If $C(k_1, ... , k_m)$ contains exactly one $3,$ then it contains at least either one more odd number $k_i$ or a pair of even numbers $k_j, k_{j'}$. From the assumption, we can build a permutation for the irreducible conjugacy classes $C(k_i, 3)$ or $C(k_j, k_{j'}, 3)$ and by Lemma \ref{stitching together} reduce to the case of no $3$'s again.\\ 

    Now, assume that $C(k_1, ... , k_m)$ contains $2,$ but doesn't contain $1$ or $3.$ Consider the permutations $(1 \ 2 \ 5 \ 4 \ 7 \ 8 \ 6 \ 3)$ and $(1 \ 2 \ 8 \ 12 \ 9 \ 6 \ 7 \ 3 \ 5 \ 10 \ 11 \ 4).$ Their commutator with the corresponding shift permutations is contained in $C(2,2,2,2)$ and $C(2,2,2,2,2,2)$ respectively. So, by Lemma \ref{stitching together}, we can find permutations for any conjugacy class of the form $C(2, ... , 2),$ that is the conjugacy class containing only and at least two pairs of $2$'s. If $C(k_1, ... , k_m)$ contains an even number of $2$'s which is greater or equal to four, we can find a permutation for the conjugacy class $C(2, ... , 2)$ with the number of $2$'s that $C(k_1, ... , k_m)$ contains and using Lemma \ref{stitching together} reduce to the case of having no $2$'s. If $C(k_1, ... , k_m)$ contains an odd number of $2$'s which is greater or equal to five, then there is at least one more even number $k_j$ and we can find permutations for the conjugacy classes $C(2, ... , 2)$ where the number of $2$'s is one less than the number of $2$'s contained in $C(k_1, ... , k_m)$ and for $C(2, k_j)$ since it is irreducible. Again using Lemma \ref{stitching together} we can reduce to the case of having no $2$'s. If $C(k_1, ... , k_m)$ contains exactly one or exactly three $2$'s, then there is at least one other even number $k_j$ and by the assumption we can find permutations for $C(2, k_j)$ or $C(2, 2, 2, k_j)$ and hence reduce to having no $2$'s. If $C(k_1, ... , k_m)$ contains exactly two $2$'s, then there is either another odd number $k_i$ or another pair of even numbers $k_j, k_{j'}.$ Note that $k_j \neq 2 \neq k_{j'}.$ By the assumpion we can find permutations for the irreducible classes $C(k_i, 2, 2)$ or $C(k_j, 2)$ and $C(k_{j'}, 2).$ We conclude that also in this case we can reduce to having no $2$'s.\\

    Now, assume that $C(k_1, ... , k_m)$ contains $1$ and $3,$ but not $2.$ Consider the permutation $(1 \ 2 \ 4 \ 3).$ Its commutator with $\sigma_4$ lies in $C(3,1).$ Let $k$ be the minimum of the number of $1$'s and the number of $3$'s contained in $C(k_1, ... , k_m).$ By Lemma \ref{stitching together}, we can find a permutation in the conjugacy class $C(3, ...  , 3, 1, ..., 1)$ with exactly $k$ $3$'s and $k$ $1$'s and reduce to the case of having either only $1$'s or only $3$'s left.\\

    Now, assume that $C(k_1, ... , k_m)$ contains $1$ and $2,$ but not $3.$ Consider the permutation $(1 \ 2 \ 6 \ 4 \ 5 \ 3).$ Its commutator with $\sigma_6$ lies in $C(2, 2, 1, 1).$ If $C(k_1, ... , k_m)$ contains numbers not equal to $1$ or $2,$ by the previous cases we can find a permutation for the conjugacy class containing those numbers and all the $2$'s. Then by Lemma \ref{adding 1}, we can also find a permutation for $C(k_1, ... , k_m).$ Otherwise, $C(k_1, ... , k_m)$ consists only of $1$'s and $2$'s. If there are more than two $2$'s we already saw that we can find a permutation for $C(2,...,2)$ and then add the $1$'s by Lemma \ref{adding 1}. If there are exactly two $2$'s, then there have to be at least two $1$'s since $n \ge 6$. In that case, we use the permutation for $C(2,2,1,1)$ constructed above and add the remaining $1$'s with Lemma \ref{adding 1}.\\

    Now, assume that $C(k_1, ... , k_m)$ contains $2$ and $3,$ but not $1.$ 
    
    Consider the permutations $(1 \ 2 \ 4 \ 6 \ 7 \ 5 \ 3),$ $(1 \ 2 \ 4 \ 6 \ 7 \ 9 \ 10 \ 5 \ 8 \ 3)$ and $(1 \ 2 \ 7 \ 5 \ 3 \ 11 \ 10 \ 4 \ 6 \ 9 \ 8).$ Their commutator with the corresponding shift permutations is contained in $C(3,2,2), \, C(3,3,2,2)$ and $C(3,2,2,2,2)$ respectively. Using these permutations, together with the ones for the conjugacy classes $C(2, ... , 2), \, C(3, ... ,3)$ and Lemma \ref{stitching together}, we can construct permutations for every conjugacy class of the form $C(2, ... , 2, 3, ... ,3)$ for any even number of $2$'s and any number of $3$'s. Hence, if $C(k_1, ... , k_m)$ contains an even number of $2$'s, we built the permutation for the corresponding conjugacy class $C(2, ... ,2,3,...,3)$ and by Lemma \ref{stitching together} are able to reduce to the case without $2$'s or $3$'s. If $C(k_1, ... , k_m)$ contains an odd number of $2$'s, then there is at least one more even number $k_j \neq 2,$ so we use permutations for $C(2, ... , 2, 3, ... , 3)$ where the number of $2$'s is one less than the number of $2$'s contained in $C(k_1, ... , k_m)$ and for $C(2, k_j)$ in order to reduce to the case without $2$'s and $3$'s.\\

    Finally, assume that $C(k_1, ... , k_m)$ contains $1, 2$ and $3.$ Without loss of generality we can assume that $k_{m'} = k_{m'+1} = ... = k_m = 1$ for some $m'$ i.e. $C(k_1, ... , k_m) = C(k_1, ... , k_{m'-1}, 1, ... ,1).$ By the above discussion, we can construct a permutation for $C(k_1, ... , k_{m'-1}).$ Hence, Lemma \ref{adding 1} yields a permutation $\pi$ for $C(k_1, ... , k_m).$   
\end{proof}

With the above Lemma \ref{irreducible suffice} in hand, it remains to show that every irreducible conjugacy class is contained in the image of the commutator map. For that, we explicitely construct permutations $\pi$ whose commutator with the shift permutation lies in the given irreducible conjugacy class. The only constraint on $\pi$ we need to fulfil is that $\pi(1) = 2.$

\subsection*{Permutations for the irreducible conjugacy classes}

For each irreducible conjugacy class below, we define a permutation $\pi \in C(n)$ for the corresponding $n$ such that $[\sigma_n, \pi]$ is contained in that conjugacy class. All permutations we define satisfy $\pi(1) = 2,$ so together with the above lemmata this section concludes the proof of Theorem \ref{commutator theorem}.\\

In the following, whenever we write $\overset{+d}{\dots}$ or $\overset{-d}{\dots}$ for some $d \in \mathbb{N}$ in a cycle of a permutation, we mean to continue the cycle by adding or subtracting $d$ consecutively until reaching the next number in the permutation. For example $(1 \ 2 \ \overset{+2}{\dots} \ 8 \ 7) = (1 \ 2 \ 4 \ 6 \ 8 \ 7).$ Whenever we write $\overset{(+d, -d')}{\dots}$ for some $d,d' \in \mathbb{N},$ we mean to continue by adding $d,$ then subtracting $d'$ and so on until reaching the next number. For example $(1 \ 2 \ \overset{(+3, -1)}{\dots} \ 9) = (1 \ 2 \ 5 \ 4 \ 7 \ 6 \ 9).$\\ 

\fbox{$C(n)$ for $n \ge 5$ odd}\\

\hspace{0.5cm} Let $\pi := (1 \ \ 2 \ \ 4 \ \ 5 \ \ \overset{+2}{\dots} \ \ n \ \ n-1 \ \ \overset{-2}{\dots} \ \ 6 \ \ 3).$ We compute:
$$[\sigma_n, \pi] = \begin{cases}
    (1 \ \ 5 \ \ \overset{+4}{\dots} \ \ n-4 \ \ n-1 \ \ \overset{-4}{\dots} \ \ 4 \ \ 7 \ \ \overset{+4}{\dots} \ \ n-2 \ \ 2 \ \ 3 \ \ n \ \ n-3 \ \ \overset{-4}{\dots} \ \ 6), & \text{if } n \equiv 1 \text{ mod } 4,\\
    (1 \ \ 5 \ \ \overset{+4}{\dots} \ \ n-2 \ \ 2 \ \ 3 \ \ n \ \ n-3 \ \ \overset{-4}{\dots} \ \ 4 \ \ 7 \ \ \overset{+4}{\dots} \ \ n-4 \ \ n-1 \ \ \overset{-4}{\dots} \ \ 6), & \text{if } n \equiv 3 \text{ mod } 4.
\end{cases}$$

In any case, we have $[\sigma_n, \pi] \in C(n).$\\

\fbox{$C(2k, 2l)$ for $k , l \ge 1$ except $k = l = 1$}\\

We subdivide this case into $k = l, k = l+1$ and $k > l + 1.$\\

For $k = l$ (and $l \ge 4$), let 
$$\pi = (1 \ \ 2 \ \ 5 \ \ 3 \ \ 4 \ \ \overset{(+3,-1)}{\dots} \ \ 2l+2 \ \  2l+6 \ \ 2l+4 \ \ 2l+5 \ \ \overset{(+3,-1)}{\dots} \ \ 4l \ \ 4l-1).$$

We compute:

$$[\sigma_{4l}, \pi] = \begin{cases}
    (1 \ \ 4 \ \ \overset{+4}{\dots} \ \ 2l \ \ 2l+5 \ \ \overset{+4}{\dots} \ \ 4l-3 \ \ 4l-2 \ \ \overset{-4}{\dots} \ \ 2l+6 \ \ 2l+3 \ \ \overset{-4}{\dots} \ \ 3)\\
    (2 \ \ \overset{+4}{\dots} \ \ 2l+2 \ \ 2l+7 \ \ \overset{+4}{\dots} \ \ 4l-1 \ \ 4l \ \ \overset{-4}{\dots} \ \ 2l+4 \ \ 2l+1 \ \ \overset{-4}{\dots} \ \ 5), & \text{if } l \equiv 0 \text{ mod } 2,\\ \\
    (1 \ \ 4 \ \ \overset{+4}{\dots} \ \ 2l+2 \ \ 2l+7 \ \ \overset{+4}{\dots} \ \ 4l-3 \ \ 4l-2 \ \ \overset{-4}{\dots} \ \ 2l+4 \ \ 2l+1 \ \ \overset{-4}{\dots} \ \ 3)\\
    (2 \ \ \overset{+4}{\dots} \ \ 2l \ \ 2l+5 \ \ \overset{+4}{\dots} \ \ 4l-1 \ \ 4l \ \ \overset{-4}{\dots} \ \ 2l+6 \ \ 2l+3 \ \ \overset{-4}{\dots} \ \ 5), & \text{if } l \equiv 1 \text{ mod } 2.
\end{cases}$$

Counting the elements in the two cycles of $[\sigma_{4l}, \pi]$ shows that in both cases $[\sigma_{4l}, \pi] \in C(2l,2l) = C(2k,2l).$ For $l \in \{2, 3\},$ the respective permutations $(1 \ 2 \ 4 \ 5 \ 8 \ 6 \ 7 \ 3)$ and $(1 \ 2 \ 5 \ 3 \ 4 \ 7 \ 6 \ 9 \ 8 \ 12 \ 10 \ 11)$ give the desired results $C(4,4)$ and $C(6,6)$.\\

For $k = l+1$ (and $l \ge 4$) let 
$$\pi = (1 \ \ 2 \ \ 5 \ \ 3 \ \ 4 \ \ \overset{(+3,-1)}{\dots} \ \ 2l+2 \ \ 2l+6 \ \ 2l+4 \ \ 2l+5 \ \ \overset{(+3,-1)}{\dots} \ \ 4l \ \ 4l-1 \ \ 4l+2 \ \ 4l+1).$$

Since this permutation is obtained by the one in the $k=l$ case by adding $4l+2, 4l+1$ at the end, which doesn't change the $(+3,-1)$ pattern, a similar computation to the one before shows $[\sigma_{4l+2}, \pi] \in C(2l + 2, 2l) = C(2k, 2l).$ For $l \in \{1, 2, 3\},$ the respective permutations $(1 \ 2 \ 5 \ 4 \ 6 \ 3) ,(1 \ 2 \ 5 \ 3 \ 4 \ 7 \ 6 \ 10 \ 8 \ 9)$ and $(1 \ 2 \ 5 \ 3 \ 4 \ 7 \ 6 \ 9 \ 8 \ 12 \ 10 \ 11 \ 14 \ 13)$ give the desired results $C(4,2), C(6,4)$ and $C(8,6).$\\

For $k > l+1$ (and $l \ge 2$), let 
$$\pi = (1 \ \ 2 \ \ 5 \ \ 3 \ \ 4 \ \ \overset{(+3,-1)}{\dots} \ \ 4l-4 \ \ 4l \ \ 4l-2 \ \ 4l-1 \ \ 4l+1 \ \ 4l+4 \ \ 4l+2 \ \ 4l+3 \ \ \overset{(+3,-1)}{\dots} \ \ 2k+2l \ \ 2k+2l-1).$$

We compute:

$$[\sigma_{2k+2l}, \pi] = \begin{cases}
    (1 \ \ 4 \ \ \overset{+4}{\dots} \ \ 4l-4 \ \ 4l-2 \ \ 4l-5 \ \ \overset{-4}{\dots} \ \ 3)\\
    (2 \ \ \overset{+4}{\dots} \ \ 4l-6 \ \ 4l-1 \ \ \overset{+4}{\dots} \ \ 2k+2l-1 \ \ 2k+2l \ \ \overset{-4}{\dots} \ \ 4l+4 \\ 4l+1 \ \ \overset{+4}{\dots} \ \ 2k+2l-3 \ \ 2k+2l-2 \ \ \overset{-4}{\dots} \ \ 4l+2 \ \ 4l \ \ 4l-3 \ \ \overset{-4}{\dots} \ \ 5), & \text{if } k+l \equiv 0 \text{ mod } 2, \\ \\
    (1 \ \ 4 \ \ \overset{+4}{\dots} \ \ 4l-4 \ \ 4l-2 \ \ 4l-5 \ \ \overset{-4}{\dots} \ \ 3)\\
    (2 \ \ \overset{+4}{\dots} \ \ 4l-6 \ \ 4l-1 \ \ \overset{+4}{\dots} \ \ 2k+2l-3 \ \ 2k+2l-2 \ \ \overset{-4}{\dots} \ \ 4l+4 \\ 4l+1 \ \ \overset{+4}{\dots} \ \ 2k+2l-1 \ \ 2k+2l \ \ \overset{-4}{\dots} \ \ 4l+2 \ \ 4l \ \ 4l-3 \ \ \overset{-4}{\dots} \ \ 5), & \text{if } k+l \equiv 1 \text{ mod } 2.
\end{cases}$$

Again counting the elements in the two cycles of $[\sigma_{2k+2l}, \pi]$ shows that in both cases $[\sigma_{2k+2l}, \pi] \in C(2k,2l).$ The remaining case is $k > l+1$ and $l=1.$ If we let $n := 2k+2,$ this corresponds to $C(n-2,2)$ for $n \ge 8.$ For this (if $n \ge 10$), let 
$$\pi = (1 \ \ 2 \ \ 4 \ \ 5 \ \ 8 \ \ 7 \ \ 10 \ \ \overset{+2}{\dots} \ \ n \ \ n-1 \ \ \overset{-2}{\dots} \ \ 9 \ \ 6 \ \ 3),$$

and compute:
$$[\sigma_n, \pi] = \begin{cases}
    (4 \ 8) \\ (1 \ \ 5 \ \ 10 \ \ \overset{+4}{\dots} \ \ n-2 \ \ 2 \ \ 3 \ \ n \ \ n-3 \ \ \overset{-4}{\dots} \ \ 9 \ \ 7 \ \ 12 \ \ \overset{+4}{\dots} \ \ n-4 \ \ n-1 \ \ \overset{-4}{\dots} \ \ 11 \ \ 6), & \text{if } n \equiv 0 \text{ mod } 4,\\ \\
    (4 \ 8) \\ (1 \ \ 5 \ \ 10 \ \ \overset{+4}{\dots} \ \ n-4 \ \ n-1 \ \ \overset{-4}{\dots} \ \ 9 \ \ 7 \ \ 12 \ \ \overset{+4}{\dots} \ \ n-2 \ \ 2 \ \ 3 \ \ n \ \ n-3 \ \ \overset{-4}{\dots} \ \ 11 \ \ 6), & \text{if } n \equiv 2 \text{ mod } 4.
\end{cases}$$

Hence, $[\sigma_n, \pi] \in C(n-2,2).$ Finally, for $n = 8$ the permutation $(1 \ 2 \ 4 \ 6 \ 8 \ 7 \ 5 \ 3)$ gives the desired conjugacy class $C(6,2).$\\

\fbox{$C(n, 3)$ for $n \ge 1$ odd}\\

\hspace{0.5cm} For $n = 1,$ $\pi = (1 \ 2 \ 4 \ 3)$ and for $n = 3,$ $\pi = (1 \ 2 \ 4 \ 3 \ 6 \ 5)$ give the desired result. For $n \ge 5,$ let $\pi = (1 \ \ 2 \ \ 4 \ \ 5 \ \ 7 \ \ 8 \ \ \overset{+2}{\dots} \ \ n + 3 \ \ n + 2 \ \ \overset{-2}{\dots} \ \ 9 \ 6 \ 3).$ We compute:

$$[\sigma_{n+3}, \pi] =  \begin{cases}
    (1 \ 5 \ 6) \, (2 \ \ 3 \ \ n+3 \ \ n \ \ \overset{-4}{\dots} \ \ 9 \ \ 4 \ \ \overset{+4}{\dots} \ \ n-1 \ \ n+2 \ \ \overset{-4}{\dots} \ \ 7 \ \ 10 \ \ \overset{+4}{\dots} \ \ n+1) & \text{if } n \equiv 1 \text{ mod } 4,\\
    (1 \ 5 \ 6) \, (2 \ \ 3 \ \ n+3 \ \ n \ \ \overset{-4}{\dots} \ \ 7 \ \ 10 \ \ \overset{+4}{\dots} \ \ n-1 \ \ n+2 \ \ \overset{-4}{\dots} \ \ 9 \ \ 4 \ \ \overset{+4}{\dots} n+1) & \text{if } n \equiv 3 \text{ mod } 4.
\end{cases}$$

In any case, we have $[\sigma_{n+3}, \pi] \in C(n,3).$\\

\fbox{$C(2k, 2l, 3)$ for $k, l \ge 1$}\\

\hspace{0.5cm} Let $n := 2k+2l$ and $\pi \in C(n)$ be any of the permutations constructed in the $C(2k, 2l)$ case above such that $[\sigma_n, \pi] \in C(2k,2l).$ Then $\pi$ is of the from $(1 \ 2 \ ... \ i \ n \ ... \ j),$ where $i, j \in \{3, ... , n-1\}.$ Define $\tau \in C(n+5)$ by 
$$\tau := (1 \ \ 2 \ \ ... \ \ i \ \ n \ \ ... \ \ j \ \ n+1 \ \ n+3 \ \ n+2 \ \ n+5 \ \ n+4),$$

where $\tau$ is the same as $\pi$ from $1$ to $j.$ Note that whatever $\pi$ is, the commutator $[\sigma_{n+5}, \tau]$ contains the cycle $(n+1 \ \ n+4 \ \ n+5).$ For any $m \in \{1, ... , n\} \setminus \{i, j\}$ it is straightforward that $[\sigma_{n+5}, \tau](m) = [\sigma_{n}, \pi](m).$ Furthermore, we also have $[\sigma_{n+5}, \tau](i) = j-1 = [\sigma_{n}, \pi](i).$ For $j$ we have $[\sigma_{n}, \pi](j) = n,$ whereas the commutator $[\sigma_{n+5}, \tau]$ contains a cycle of the form $(... \ \ j \ \ n+2 \ \ n+3 \ \ n \ \ ...).$ Since $[\sigma_{n}, \pi] \in C(2k, 2l)$ and one can check that the element $n$ is contained in the $2k$ cycle of $[\sigma_{n}, \pi]$ for any permutation $\pi$ constructed above, we conclude that $[\sigma_{n+5}, \tau] \in C(2k+2, 2l, 3).$\\ 

Using the permutations for the $C(2k,2l)$ case above and the method just described, we cover every case except for $C(4, 2, 3)$ and $C(2k, 2k, 3)$ for $k \ge 1.$ The permutation $(1 \ 2 \ 4 \ 5 \ 8 \ 7 \ 9 \ 6 \ 3)$ gives the desired result for $C(4, 2, 3).$ For $C(2k, 2k , 3)$ (and $k \ge 4$) consider 
$$\pi = (1 \ \ 2 \ \ 6 \ \ 3 \ \ 4 \ \ 10 \ \ 5 \ \ 8 \ \ 9 \ \ 7 \ \ 12 \ \ \overset{(-1, +3)}{\dots} \ \ 4k-3 \ \ 4k-1 \ \ 4k+1 \ \ 4k \ \ 4k+3 \ \ 4k+2).$$  

We compute

\begin{align*}
    [\sigma_{4k+3}, \pi] = (1 \ \ 5 \ \ 7 \ \ 13 \ \ \overset{+4}{\dots} \ \ 4k-3 \ \ 4k \ \ 4k+1 \ \ 4k-2 \ \ \overset{-4}{\dots} \ \ 10)\\
    (2 \ \ 8 \ \ 3 \ \ 9 \ \ 4 \ \ 11 \ \ \overset{+4}{\dots} \ \ 4k-5 \ \ 4k-4 \ \ \overset{-4}{\dots} \ \ 12 \ \ 6)\\
    (4k-1 \ \ 4k+2 \ \ 4k+3)
\end{align*}

and see that $[\sigma_{4k+3}, \pi] \in C(2k, 2k, 3).$\\

Finally, the permutations $(1 \ 2 \ 4 \ 6 \ 7 \ 5 \ 3), \, (1 \ 2 \ 4 \ 6 \ 5 \ 3 \ 7 \ 9 \ 8 \ 11 \ 10)$ and $(1 \ 2 \ 6 \ 3 \ 4 \ 10 \ 5 \ 8 \ 9 \ 7 \ 11 \ 13 \ 12 \ 15 \ 14)$ give the desired result for the conjugacy classes $C(2, 2, 3), C(4, 4, 3)$ and $C(6, 6, 3)$ respectively.\\

\fbox{$C(n, 2, 2)$ for $n \ge 3$ odd}\\

\hspace{0.5cm} For $n =3,$ $\pi = (1 \ 2 \ 4 \ 6 \ 7 \ 5 \ 3)$ gives the desired result. For $n \ge 5,$ we let $\pi = (1 \ \ 2 \ \ 4 \ \ n+2 \ \ n+4 \ \ n+3 \ \ 5 \ \ 6 \ \ \overset{+2}{\dots} \ \ n+1 \ \ n \ \ \overset{-2}{\dots} \ \ 7 \ \ 3).$ We compute:
$$[\sigma_{n+4}, \pi] = \begin{cases}
    (2 \ n+2) \, (4 \ n+3) \\ (1 \ \ 6 \ \ \overset{+4}{\dots} \ \ n-3 \ \ n \ \ \overset{-4}{\dots} \ \ 5 \ \ 8 \ \ \overset{+4}{\dots} \ \ n-1 \ \ 3 \ \ n+4 \ \ n+1 \ \ n-2 \ \ \overset{-4}{\dots} \ \ 7) & \text{if } n \equiv 1 \text{ mod } 4,\\ \\
    (2 \ n+2) \, (4 \ n+3) \\ (1 \ \ 6 \ \ \overset{+4}{\dots} \ \ n-1 \ \ 3 \ \ n+4 \ \ n+1 \ \ n-2 \ \ \overset{-4}{\dots} \ \ 5 \ \ 8 \ \ \overset{+4}{\dots} \ \ n-3 \ \ n \ \ \overset{-4}{\dots} \ \ 7) & \text{if } n \equiv 3 \text{ mod } 4.
\end{cases}$$

In any case, we have $[\sigma_{n+4}, \pi] \in C(n, 2, 2).$\\

\fbox{$C(2k, 2, 2, 2)$ for $k \ge 1$}\\

\hspace{0.5cm} For $k = 1,$ $\pi = ( 1 \ 2 \ 5 \ 4 \ 7 \ 8 \ 6 \ 3)$ and for $k = 2,$ $\pi = (1 \ 2 \ 4 \ 8 \ 9 \ 7 \ 5 \ 6 \ 10 \ 3)$ give the desired result. For $k \ge 3,$ we distinguish whether $k \equiv 0 \ (\text{mod } 2)$ or $k \equiv 1 \ (\text{mod } 2).$ In the first case, let $\pi = (1 \ \ 2 \ \ 5 \ \ 4 \ \ 7 \ \ 9 \ \ 10 \ \ \overset{+2}{\dots} \ \ n-2 \ \ n-3 \ \ \overset{-2}{\dots} \ \ 11 \ \ 8 \ \ n-1 \ \ n \ \ 6 \ \ 3).$ In the second case, let $\pi = (1 \ \ 2 \ \ 5 \ \ 4 \ \ 7 \ \ \overset{+2}{\dots} \ \ n-3 \ \ n-2 \ \ \overset{-2}{\dots} \ \ 8 \ \ n-1 \ \ n \ \ 6 \ \ 3).$ Here, $n := 2k + 6.$ We compute:

$$[\sigma_n, \pi] = \begin{cases}
    (1 \ 5) \, (2 \ n-1) \, (3 \ n) \\ (4 \ \ 10 \ \ \overset{+4}{\dots} \ \ n-4 \ \ 7 \ \ 8 \ \ n-2 \ \ n-5 \ \ \overset{-4}{\dots} \ \ 9 \ \ 12 \ \ \overset{+4}{\dots} \ \ n-6 \ \ n-3 \ \ \overset{-4}{\dots} \ \ 11 \ \ 6) & \text{if } k \equiv 0 \text{ mod } 2,\\ \\
    (1 \ 5) \, (2 \ n-1) \, (3 \ n) \\ (4 \ \ 9 \ \ \overset{+4}{\dots} \ \ n-3 \ \ 7 \ \ \overset{+4}{\dots} \ \ n-5 \ \ n-4 \ \ \overset{-4}{\dots} \ \ 8 \ \ n-2 \ \ \overset{-4}{\dots} \ \ 6) & \text{if } k \equiv 1 \text{ mod } 2.
\end{cases}$$

In any case, we have $[\sigma_n, \pi] \in C(2k, 2, 2, 2).$\\

\printbibliography

@article{LJ,
    author = {Luke Jeffreys},
    title = {Single-cylinder square-tiled surfaces and the ubiquity of ratio-optimising pseudo-Anosovs},
    journal = {Transactions of the american mathematical society, Vol. 374, Nb. 8, pp. 5739-5781},
    year = {2021}
}

@article{O,
    author = {Oystein Ore},
    title = {Some remarks on commutators},
    journal = {Proc. Amer. Math. Soc. 2, 307-314},
    year = {1951}
}

@article{LOST,
    author = {Liebeck, Martin Walter and O'Brien, Eamonn Anthony and Shalev, Aner and Tiep, Pham Huu},
    title = {The Ore conjecture},
    journal = {J. Eur. Math. Soc. 12, 939–1008},
    year = {2009}
}

\end{document}